\documentclass[12pt,leqno]{amsart}
\newtheorem{thm}{Theorem}[section]
\newtheorem{defi}{Definition}[section]

\newtheorem{Note}{Note}[section]
\newcommand{\be}{\begin{equation}}
\newcommand{\ee}{\end{equation}}
\numberwithin{equation}{section}
\newcommand{\bea}{\begin{eqnarray}}
\newcommand{\eea}{\end{eqnarray}}
\newcommand{\beb}{\begin{eqnarray*}}
\newcommand{\eeb}{\end{eqnarray*}}
\usepackage{amssymb,amsfonts,amsthm,setspace,indentfirst}
\usepackage[dvips]{graphics}
\usepackage{epsfig}
\begin{document}
\title{Paracompactness in a bispace}
\author{Amar Kumar Banerjee$^{1}$ and Rahul Mondal$^{2}$}
\address{\noindent\newline Department of Mathematics,\newline The University of 
Burdwan, \newline Golapbag, Burdwan-713104,\newline West Bengal, India.}   
\email{akbanerjee@math.buruniv.ac.in, akbanerjee1971@gmail.com}
\email{imondalrahul@gmail.com}
\begin{abstract}
The idea of pairwise paracompactness was studied by many authors in a bitopological space. Here we study the same in the setting of more general structure of a bispace using the thoughts of the same given by Bose et al \cite{MKBOSE}.
\end{abstract}
\noindent\footnotetext{$\mathbf{2010}$\hspace{5pt}AMS\; Subject\; Classification: 54A05, 54E55, 54E99.}
\maketitle
\vspace{0.5in}
\section{\bf{Introduction}}
The idea of paracompactness given by $\text{Dieudonn}e^{'}$ in the year 1944 came out as a generalization of the notion of compactness. It has many implication in field of differential geometry and it plays important roll in metrization theory. 
The concept of the Alexandroff space \cite{ASF} $($i.e., a $\sigma$-space or simply a space$)$ was introduced by A. D. Alexandroff in the year 1940 as a generalization of a topological space where the union of open sets were taken to be open for only countable collection of open sets instead of arbitrary collection. 
Another kind of generalization of a topological space  is the idea of a bitopological space introduced by J.C. Kelly in \cite{JCK}. Using these ideas Lahiri and Das \cite{BN} introduced the idea of a bispace as a generalization of a $\sigma$-space. More works on topological properties were carried out by many authors $( $\cite{WJP}, \cite{ILR}, \cite{JW} etc.$)$ in the setting of a bitopological space and in the setting of a bispace $($\cite{AKB2}, \cite{AKB3}, \cite{AKB4}, \cite{RM}, \cite{RM3} etc.$)$. Datta \cite{MCD} studied the idea of paracompactness in a bitopological space and  tried to get analogous results of topological properties given by Michael \cite{Michael} in respect of paracompactness. In 1986 Raghavan and Reilly \cite{TGR} gave the idea of paracompactness in a bitopological space in another way. Later in 2008 M. K. Bose et al \cite{MKBOSE} studied the same in a bitopological space as a generalization of pairwise compactness. Here we study the idea of pairwise paracompactness using the thoughts given by Bose et al \cite{MKBOSE} and investigate some results of it in the setting of more general structure a bispace.
\section{\bf{Preliminaries}}\label{preli}
\begin{defi}\cite{ASF} 
A set $X$ is called an Alexandroff space or $\sigma $- space or simply space if in it is chosen a system $\mathcal {F}$ of subsets of $X$, satisfying the following axioms\\
$\left(i\right)$The intersection of countable number sets in $\mathcal {F}$ is a set in $\mathcal {F}$.\\
$\left(ii\right)$The union of finite number of sets from $\mathcal {F}$ is a set in $\mathcal {F}$.\\
$\left(iii\right)$The void set and $X$ are in $\mathcal {F}$.\end{defi}
Sets of $\mathcal F$ are called closed sets. There complementary sets are called open.It is clear that instead of closed sets in the definition of a space, one may put open sets with  subject to the conditions of  countable summability, finite intersectability and the condition that $X$ and the void set should be open.\\ 
The collection of such open will sometimes be denoted by $\mathcal{P}$ and the space by $(X, \mathcal{P})$.
It is noted that $\mathcal{P}$ is not a topology in general as can be seen by taking $X=\mathbb{R}$, the set of real numbers and $\tau$ as the collection of all $F_\sigma$ sets in $\mathbb{R}$.\\
\begin{defi}\cite{ASF} 
To every set $M$ we correlate its closure $\overline M$ as the intersection of all closed sets containing M.\end{defi}  
Generally the closure of a set in a $\sigma$-space is not a closed set. We denote the closure of a set $M$ in a space $(X, \mathcal{P})$ by $\mathcal{P}$-cl$(M)$ or cl$(M)$ or simply $\overline M$ when there is no confusion about $\mathcal{P}$.
The idea of limit points, derived set, interior of a set etc. in a space are similar as in the case of a topological space which have been thoroughly discussed in \cite{PD}.
\begin{defi} \cite{AKB}
Let $(X, \mathcal{P})$ be a space. A family of open sets $B$ is said to form a base $($open$)$ for $\mathcal{P}$ if and only if every open set  can be expressed as countable union of members of $B$.\end{defi}
\begin{thm} \cite{AKB}
 A collection of subsets $B$ of a set $X$ forms an open base of a suitable space structure $\mathcal{P}$ of $X$ if and only if \\
\noindent 1) the null set $\phi \in B$\\
\noindent 2) $X$ is the countable union of some sets belonging to $B$.\\
\noindent 3) intersection of any two sets belonging to $B$ is expressible as countable union of some sets belonging to $B$.\\
\end{thm}
\begin{defi} \cite{BN}
Let $X$ be a non-empty set. If $\mathcal{P}$ and $\mathcal{Q}$ be two collection of subsets of $X$ such that $(X,\mathcal{P})$ and $(X,\mathcal{Q})$ are two spaces, then $X$ is called a bispace. 
\end{defi}
\begin{defi} \cite{BN} 
A bispace $(X,\mathcal{P},\mathcal{Q})$ is called pairwise $T_{1}$ if for any two distinct points $x$, $y$ of $X$, there exist $U\in \mathcal{P}$ and $V\in\mathcal{Q}$ such that $x\in U$, $y\notin U$ and $y\in V$, $x\notin V$.
\end{defi}
\begin{defi} \cite{BN}
A bispace $(X,\mathcal{P},\mathcal{Q})$ is called pairwise Hausdorff if for any two distinct points $x$, $y$ of $X$, there exist $U\in\mathcal{P}$ and $V\in\mathcal{Q}$ such that $x\in U$, $y\in V$, $U\cap V=\phi$. 
\end{defi}
\begin{defi} \cite{BN} 
In a bispace $(X,\mathcal{P},\mathcal{Q})$, $\mathcal{P}$ is said to be regular with respect to $\mathcal{Q}$ if for any $x\in X$ and a $\mathcal{P}$-closed set $F$ not containing $x$, there exist $U\in \mathcal{P}$, $V\in\mathcal{Q}$ such that $x\in U$, $F\subset V$, $U\cap V=\phi$. $(X,\mathcal{P},\mathcal{Q})$ is said to be pairwise regular if $\mathcal{P}$ and $\mathcal{Q}$ are regular with respect to each other. 
\end{defi}
\begin{defi} \cite{BN} 
A bispace $(X,\mathcal{P},\mathcal{Q})$ is said to be pairwise normal if for any $\mathcal{P}$-closed set $F_{1}$ and $\mathcal{Q}$-closed set $F_{2}$ satisfying $F_{1}\cap F_{2}=\phi$, there exist  $G_{1}\in \mathcal{P}$, $G_{2}\in \mathcal{Q}$ such that $F_{1}\subset G_{2}$, $F_{2}\subset G_{1}$, $G_{1}\cap G_{2}=\phi$.
\end{defi}
\section{\bf{Pairwise paracompactness}}
\noindent We called a space $($ or a set $)$ is bicompact \cite{BN} if every open cover of it has a finite subcover. Also similarly as \cite{BN} a cover B of $(X,\mathcal{P},\mathcal{Q})$ is said to be pairwise open if $B \subset \mathcal{P} \cup \mathcal{Q}$ and B contains at least one nonempty member from each of $\mathcal{P}$ and $\mathcal{Q}$. Bourbaki and many authors defined the term paracompactness in a topological space including the requirement that the space is Hausdorff. Also in a bitopological space some authors follow this idea. But in our discussion we shall follow the convention as adopted in Munkresh\cite{JRM} to define the following terminologies as in the case of a topological space.
\begin{defi} (cf.\cite{JRM})
In a space $X$ a collection of subsets $\mathcal{A}$ is said to be locally finite in $X$ if every point has a neighborhood that intersects only a finitely many elements of $\mathcal{A}$.
\end{defi} 
Similarly a collection of subsets $\mathcal{B}$ in a space $X$ is said to be countably locally finite in $X$ if $\mathcal{B}$ can be expressed as a countable union of locally finite collection. 
\begin{defi} (cf.\cite{JRM})
Let $\mathcal{A}$ and $\mathcal{B}$ be two covers of a space $X$, $\mathcal{B}$ is said to be a refinement of $\mathcal{A}$ if for $B \in \mathcal{B}$ there exists a $A \in \mathcal{A}$ containing $B$.
\end{defi} 
We call $\mathcal{B}$ an open refinement of $\mathcal{A}$ if the elements of $\mathcal{B}$ are open and similarly if the elements of $\mathcal{B}$ are closed  $\mathcal{B}$ is said to be a closed refinement.
\begin{defi} (cf.\cite{JRM})
A space $X$ is said to be paracompact if every open covering $\mathcal{A}$ of $X$ has a locally finite open refinement $\mathcal{B}$ that covers $X$.
\end{defi}
As in the case of a topological space \cite{MCD, MKBOSE} we define the following terminologies. Let $\mathcal{A}$ and $\mathcal{B}$ be two pairwise open covers of a bispace $(X,\mathcal{P},\mathcal{Q})$. Then $\mathcal{B}$ is said to be a parallel refinement \cite{MCD} of $\mathcal{A}$ if for any $\mathcal{P}$-open set$(\text{respectively } \mathcal{Q} \text{-open set})$ $B$ in $\mathcal{B}$ there exists a $\mathcal{P}$-open set$(\text{respectively } \mathcal{Q} \text{-open set})$ $A$ in $\mathcal{A}$ containing $B$. Let $\mathcal{U}$ be a pairwise open cover in a bispace $(X, \mathcal{P}_{1}, \mathcal{P}_{2})$. If $x$ belongs to $X$ and $M$ be a subset of $X$, then by ``$M$ is $\mathcal{P}_{\mathcal{U} x}$-open" we mean $M$ is $\mathcal{P}_{1}$-open$(\text{respectively } \mathcal{P}_{2} \text{-open set})$ if $x$ belongs to a $\mathcal{P}_{1}$-open set$(\text{respectively } \mathcal{P}_{2} \text{-open set})$ in $\mathcal{U}$.
\begin{defi} (cf. \cite{MKBOSE})
Let $\mathcal{A}$ and $\mathcal{B}$ be two pairwise open covers of a bispace $(X, \mathcal{P}_{1}, \mathcal{P}_{2})$. Then $\mathcal{B}$ is said to be a locally finite refinement of $\mathcal{A}$ if for each $x$ belonging to $X$, there exists a $\mathcal{P}_{\mathcal{A} x}$-open open neighborhood of $x$ intersecting only a finite number of sets of $\mathcal{B}$.
\end{defi}
\begin{defi} (cf. \cite{MKBOSE})
A bispace $(X, \mathcal{P}_{1}, \mathcal{P}_{2})$ is said to be pairwise paracompact if every pairwise open cover of $X$ has a locally finite parallel refinement.
\end{defi}
\noindent To study the notion of paracompactness in a bispace the idea of pairwise regular and strongly pairwise regular spaces plays significant roll as discussed below.\\
\indent As in the case of a bitopological space a bispace $(X, \mathcal{P}_{1}, \mathcal{P}_{2})$ is said to be strongly pairwise regular\cite{MKBOSE} if $(X, \mathcal{P}_{1}, \mathcal{P}_{2})$ is pairwise regular and both the spaces $(X, \mathcal{P}_{1})$ and $(X, \mathcal{P}_{2})$ are regular.\\
\indent Now we present two examples, the first one is of a strongly regular bispace and the second one is of a pairwise regular bispace without being a strongly pairwise regular bispace.\\ 
\textbf{Example 3.1.}
Let $X=\mathbb{R}$ and $(x,y)$ be an open interval in $X$. We consider the collection $\tau_{1}$ with sets $A$ in $\mathbb{R}$ such that either $(x,y) \subset \mathbb{R} \setminus A$ or $A \cap (x,y)$ can be expressed as some union of open subintervals of $(x,y)$ and $\tau_{2}$ be the collection of all countable subsets in $(x,y)$. Also if $\tau$ be the collection of all countable union of members of $\tau_{1} \cup \tau_{2}$ then clearly $(X,\tau)$ is a $\sigma$-space but not a topological space. Also consider the bispace $(X, \tau, \sigma)$, where $\sigma$ is the usual topology on $X$. \\
\indent  We first show that $(X,\tau)$ is regular. Let $p \in X$ and $P$ be any $\tau$-closed set not containing $p$. Then $A =\{p \}$ is a $\tau$-open set containing $p$. Also $A =\{p \}$ is closed in $(X,\tau)$ because if $p \notin (x, y)$ then  $A^{c} \cap (x,y)= (x,y)$ and if $p \in (x, y)$ then $A^{c} \cap (x,y)= (x,p) \cup (p,y)$ and hence $A^{c}$ is a $\tau$-open set containing $P$. \\
\indent Now we show that the bispace $(X, \tau , \sigma)$ is pairwise regular. Let $p\in X$ and $M$ be a $\tau$-closed set not containing $p$. Then $A =\{p \}$ is a $\tau$-open set containing $p$ and also as every singleton set is closed in $(X,\sigma)$, $A^{c}$ is a $\sigma$-open set containing $M$.\\ 
 \indent Now let $p \in X$ and $P$ be a $\sigma$-closed set not containing $p$. Now consider the case when $P \cap (x,y)=\phi$ then $P$ is a $\tau$-open set containing $P$ and $P^{c}$ is a $\sigma$-open set containing $p$.\\
\indent Now we consider the case when $P \cap (x, y) \neq \phi$. Since $p \notin P$, $P^{c}$ is a $\sigma$-open set containing $p$ and hence there exists an open interval $I$ containing $p$ be such that $p \in I \subset P^{c}$ and $p \in \overline{I} \subset P^{c}$, where $\overline{I}$ is the closer of $I$ with respect to $\sigma$. If $I$ intersects $(x,y)$ then let $I_{1} = (x,y) \setminus \overline{I}$. Clearly $I_{1}$ is non empty because $ P \cap (x,y) \neq \phi$. Also $\overline{I} \subset P^{c}$ and hence $(x,y) \setminus P^{c} \subset (x,y) \setminus \overline{I}$ and its follows that $P \cap (x, y) \subset I_{1}$. So clearly $P \cup I_{1}$ is a $\tau$-open set containing $P$ and $I$ is a $\sigma$-open set containing $p$ and which are disjoint. Again if $I$ does not intersect $(x,y)$ then $P \cup (x,y)$ is a $\tau$-open set containing $P$ and $I$ itself a $\sigma$-open set containing $p$ and which are disjoint. Therefore the bispace $(X, \tau , \sigma)$ is strongly pairwise regular.\\ 
\noindent \textbf{Example 3.2.}
Let $X= \mathbb{R}$ and $(X, \tau_{1}, \tau_{2})$ be a bispace where $(X, \tau_{1})$ is cocountable topological space and $\tau_{2}= \{ X, \phi \} \cup \{ \text{countable subsets of real numbers} \}$. Clearly $\tau_{2}$ is not a topology and hence $(X, \tau_{1}, \tau_{2})$ is not a bitopological space. We show that $(X, \tau_{1}, \tau_{2})$ is a pairwise regular bispace but not a strongly pairwise regular bispace. Let $p \in X$ and $A$ be a $\tau_{1}$-closed set not containing $p$. Then clearly $A$ itself a $\tau_{2}$-open set containing $A$ and $A^{c}$ is a $\tau_{1}$-open set containing $p$ and clearly they are disjoint.\\
\indent Similarly if $B$ is a $\tau_{2}$-closed set such that $p \notin B$, then $B$ being a complement of a countable set is $\tau_{1}$-open set containing $B$. Also $B^{c}$ being countable is $\tau_{2}$-open set containing $p$.\\
\indent Now let $p \in X$ and $P$ be a closed set in $(X, \tau_{2})$ such that $p \notin P$. Then $P$ must be a complement of a countable set in $\mathbb{R}$ and hence it must be a uncountable set. So clearly the only open set containing $P$ is $\mathbb{R}$ itself. Therefore $(X, \tau_{2})$ is not regular and hence $(X, \tau_{1}, \tau_{2})$ can not be strongly pairwise regular.
\begin{Note}
In a bitopological space, pairwise Hausdorffness and pairwise paracompactness together imply pairwise normality but similar result holds in a bispace if an additional condition C$(1)$ holds. 
\end{Note}
\begin{thm}
Let $(X,\mathcal{P},\mathcal{Q})$ be a bispace, which is pairwise Hausdorff and pairwise paracompact and satisfies the condition C$(1)$ as stated below then it is pairwise normal.\end{thm}
\noindent C$(1):$ If $A\subset X$ is expressible as an arbitrary union of $\mathcal{P}$-open sets and $A \subset B$, $B$ is an arbitrary intersection of $\mathcal{Q}$-closed sets, then there exists a $\mathcal{P}$-open set $K$, such that $A \subset K \subset B$, the role of $\mathcal{P}$ and $\mathcal{Q}$ can be interchangeable.
\begin{proof}
We first show that $X$ is pairwise regular. So let us suppose $F$ be a $\mathcal{P}$-closed set not containing $x \in X$. Since $X$ is pairwise Hausdorff for $\xi \in F$, there exists a $U_{\xi}\in \mathcal{P}$ and $V_{\xi}\in \mathcal{Q}$, such that $x \in U_{\xi}$ and $\xi \in V_{\xi}$ and $U_{\xi} \cap V_{\xi} = \phi$. Then the collection $\{V_{\xi}   : \xi \in F\} \cup (X \setminus F)$ forms a pairwise open cover of $X$. Therefore it has a locally finite parallel refinement $\mathcal{W}$. Let $H= \cup \{W \in \mathcal{W} : W\cap F\neq \phi \}$. Now $x\in X \setminus F$ and $X \setminus F$ is $\mathcal{P}$-open set and hence there exists a $\mathcal{P}$-open neighborhood $D$ of $x$ intersecting only a finite number of members $W_{1}, W_{2} \dots W_{n}$ of $\mathcal{W}$. Now if $W_{i} \cap F = \phi$ for all $n=1, 2 \dots n$, then $H \cap D = \phi$. Therefore by C$(1)$ we must have a $\mathcal{Q}$-open set $K$ such that $F \subset H \subset K \subset D^{c}$. Hence we have a $\mathcal{Q}$-open set $K$ containing $F$ and $\mathcal{P}$-open set $D$ containing $x$ with $D \cap K= \phi$. If there exists a finite number of elements $W_{p_{1}}, W_{p_{2}} \dots W_{p_{k}}$ from the collection $\{W_{1}, W_{2} \dots W_{n} \}$ such that $W_{p_{i}} \cap F \neq \phi$, then we consider $V_{\xi _{p_{i}}}$ such that $W_{p_{i}} \subset V_{\xi _{p_{i}}}$, $\xi _{p_{i}} \in F$ and $i=1, 2 \dots k$, since $\mathcal{W}$ is a locally finite parallel refinement of $\{V_{\xi}   : \xi \in F\} \cup (X \setminus F)$. Now if $U_{\xi_{p_{i}}}$ are the corresponding member of $V_{\xi_{p_{i}}}$, then $x \in D \cap (\bigcap^{n}_{i=1}U_{\xi_{p_{i}}})= G(say) \in \mathcal{P}$. Since $\mathcal{W}$ is a cover of $X$ it covers also $D$ and since $D$ intersects only finite number of members $W_{1}, W_{2} \dots W_{n}$, these $n$ sets covers $D$. Now since the members $W_{p_{1}}, W_{p_{2}} \dots W_{p_{k}}$ be such that $W_{p_{i}} \cap F \neq \phi$, we have $D \cap F \subset \bigcup_{i=1}^{k}W_{p_{i}}$. Now let $W_{p_{i}} \subset V_{\xi_{p_{i}}}$ for some $\xi_{p_{i}} \in F$ and consider $U_{\xi_{p_{i}}}$ corresponding to $V_{\xi_{p_{i}}}$ be such that $U_{\xi_{p_{i}}} \cap V_{\xi_{p_{i}}}= \phi$. Now we claim that $G \cap F = \phi$. If not let $y \in G \cap F = [D \cap (\bigcap^{n}_{i=1}U_{\xi_{p_{i}}})] \cap F= [D \cap F]\cap (\bigcap^{n}_{i=1}U_{\xi_{p_{i}}})$. Then $y \in D \cap F$ and hence there exists $W_{p_{i}}$ for some $i=1,2, \dots k$ such that $y \in W_{p_{i}} \subset V_{\xi_{p_{i}}}$. Also $y \in (\bigcap^{n}_{i=1}U_{\xi_{p_{i}}}) \subset U_{\xi_{p_{i}}}$ and hence $y \in  U_{\xi_{p_{i}}} \cap V_{\xi_{p_{i}}}$, which is a contradiction. So $G \cap F = \phi$. Now we have a $\mathcal{P}$-open neighborhood $G$ of $x$ intersecting only a finite number of members $W_{r_{1}}, W_{r_{2}} \dots W_{r_{k}}$ of $\mathcal{W}$ where $W_{r_{i}} \cap F= \phi$. So by similar argument there exists a $\mathcal{Q}$-open set $K$ such that $F \subset H \subset K \subset G^{c}$. Thus we have a $\mathcal{Q}$-open set $K$ containing $F$ and a $\mathcal{P}$-open set $G$ containing $x$ such that $G \cap K=\phi$. \\
\indent Next let $A$ be a $\mathcal{Q}$-closed set and $B$ be a $\mathcal{P}$-closed set and $A\cap B =\phi$. Then for every $x \in B$ and $\mathcal{Q}$-closed set $A$ there exists $\mathcal{P}$-open set $U_{x}$ containing $A$ and $\mathcal{Q}$-open set $V_{x}$ containing $x$ with $U_{x} \cap V_{x}= \phi$. Now the collection $\mathcal{U}= (X \setminus B) \cup \{V_{x} : x \in B \}$ forms a pairwise open cover of $X$. Hence there exists a locally finite parallel refinement $\mathcal{M}$ of $\mathcal{U}$. Clearly $B \subset Q$ where $Q= \cup \{ M \in \mathcal{M} : M\cap B \neq \phi \}$. Now for $x \in X \setminus B$, a $\mathcal{P}$-open set there exists a $\mathcal{P}$-open neighborhood of $x$ intersecting only a finite number of elements of $\mathcal{M}$. Since $A \subset X \setminus B$, so for $x \in A$ there exists a $\mathcal{P}$-open neighborhood $D_{x}$ of $x$ intersecting only a finite number of elements $M_{x_{1}}, M_{x_{2}} \dots M_{x_{n}}$ of $\mathcal{M}$ with $M_{x_{i}} \cap B\neq \phi$ for some $i=1,2,\dots n$. Suppose if $M_{x_{i}} \subset V_{x_{i}}$, $i=1, 2 \dots n$ and let $P_{x}= D_{x} \cap (\bigcap^{n}_{i=1}U_{x_{i}})$ where $U_{x_{i}} \cap V_{x_{i}}= \phi$. If $M_{x_{i}} \cap B= \phi$ for all $i=1,2,\dots n$, then we consider $P_{x}= D_{x}$. Now if $P= \bigcup \{P_{x} : x\in A\}$ then $A\subset P$ and $P \subset Q^{c}$.\\
\indent	Now by the given condition C$(1)$ there exists a $\mathcal{P}$-open set $R$ be such that $A \subset P \subset R \subset Q^{c}$. Again by the same argument there exists a $\mathcal{Q}$-open set $S$ be such that $B \subset Q \subset S \subset R^{c}$. Hence there exists a $\mathcal{P}$-open set $R$ containing $A$ and $\mathcal{Q}$-open set $S$ containing $B$ with $R \cap S= \phi$.
\end{proof}
\begin{thm}
If the bispace $(X, \mathcal{P}_{1}, \mathcal{P}_{2})$ is strongly pairwise regular and satisfies the condition C$(2)$ given below, then the following statements are equivalent:\\
$(i)$ X is pairwise paracompact.\\
$(ii)$ Each pairwise open cover $\mathcal{C}$ of $X$ has a countably locally finite parallel refinement.\\
$(iii)$ Each pairwise open cover $\mathcal{C}$ of $X$ has a locally finite refinement.\\
$(iv)$  Each pairwise open cover $\mathcal{C}$ of $X$ has a locally finite refinement $\mathcal{B}$ such that if $B \subset C$ where $B \in \mathcal{B}$ and $C \in \mathcal{C}$, then $\mathcal{P}_{1}$-cl$(B) \cup \mathcal{P}_{2}$-cl$(B) \subset C$.\\
\end{thm}
\noindent C$(2) :$ If $M \subset X$ and $\mathcal{B}$ is a subfamily of $\mathcal{P}_{1} \cup \mathcal{P}_{2}$ such that $\mathcal{P}_{i}$-$cl(B) \cap M=\phi$, for all $B \in \mathcal{B}$, then  there exists a $\mathcal{P}_{i}$- open set $S$ such that $M \subset S \subset[\bigcup_{B\in \mathcal{B}}\mathcal{P}_{i}$-$cl(B)]^{c}$.
\begin{proof}
$(i) \Rightarrow (ii)$\\
Let $\mathcal{C}$ be a pairwise open cover of $X$. Let $\mathcal{U}$ be a locally finite parallel refinement of $\mathcal{C}$. Then the collection  $\mathcal{V}= \bigcup _{n=1} ^{\infty} \mathcal{V}_{n}$, where $\mathcal{V}_{n} = \mathcal{U}$ for all $n\in \mathbb{N}$, becomes the countably locally finite parallel refinement of $\mathcal{C}$.\\
$(ii) \Rightarrow (iii)$\\
We consider a pairwise open cover $\mathcal{C}$ of $X$. Let $\mathcal{V}$ be a parallel refinement of $\mathcal{C}$, such that $\mathcal{V}= \bigcup _{n=1} ^{\infty} \mathcal{V}_{n}$, where for each $n$ and for each $x$ there exists a $\mathcal{P}_{\mathcal{C}x}$-open neighborhood of $x$ intersecting only a finite number of members of $\mathcal{V}_{n}$. For each $n \in \mathbb{N}$, let us agree to write $\mathcal{V}_{n}$ as $\mathcal{V}_{n}= \{\mathcal{V}_{n \alpha} : \alpha \in \wedge_{n} \}$ and we consider $M_{n} = \bigcup _{\alpha \in \wedge_{n}} \mathcal{V}_{n \alpha}$, $n \in \mathbb{N}$. Clearly the collection $\{ M_{n} \}_{n \in \mathbb{N}}$ is a cover of $X$. Let $N_{n}= M_{n} - \bigcup _{k<n} M_{k}$. Clearly for $x \in X$ if $x \in M_{n}$, where $n$ is the least positive integer then $x \in N_{n}$ and hence $\{ N_{n} : n \in \mathbb{N}\}$ covers $X$. Also $N_{n} \subset M_{n}$ for every $n$, so $\{ N_{n} : n \in \mathbb{N}\}$ is a refinement of  $\{ M_{n} : n \in \mathbb{N}\}$. The family  $\{ N_{n} : n \in \mathbb{N}\}$ is locally finite because for $x \in X$ there exists a $\mathcal{V}_{n \alpha} \in \mathcal{V}$ which can intersects only some or all of $N_{1}, N_{2} \dots N_{n}$. Now the collection $\{ \mathcal{V}_{n \alpha} \cap N_{n} : \alpha \in \wedge_{n}, n \in \mathbb{N} \}$ covers $X$ as if $x \in \mathcal{V}_{p \alpha}$ for the least positive integer $p$ then $x \in N_{p}$ and hence $x \in \mathcal{V}_{p \alpha} \cap N_{p}$. So clearly $\{ \mathcal{V}_{n \alpha} \cap N_{n} : \alpha \in \wedge_{n}, n \in \mathbb{N} \}$ is a refinement of $\mathcal{V}$ and hence of $\mathcal{C}$. Also for $x \in X$ there exists a $\mathcal{P}_{\mathcal{C}x}$-open neighborhood $\mathcal{V}_{ k \alpha}$ intersecting only a finite number of members of $\{ N_{n} : n \in \mathbb{N}\}$ and hence it intersects only a finite number of members of $\{ \mathcal{V}_{n \alpha} \cap N_{n} : \alpha \in \wedge_{n}, n \in \mathbb{N} \}$.\\
$(iii) \Rightarrow (iv)$\\
Let $\mathcal{C}$ be a pairwise open cover of $X$. Let $x \in X$ and suppose that $x \in C_{x}$ for some $C_{x} \in \mathcal{C}$. Without any loss of generality let $C_{x} \in \mathcal{P}_{1}$. Then $x \notin C_{x}^{c}$ and hence by using the condition of strongly pairwise regularity of $X$ there exists a $\mathcal{P}_{1}$-open set $D_{1}$ containing $x$ and a $\mathcal{P}_{1}$-open set $D_{1}^{'}$ containing $C_{x}^{c}$ with $D_{1} \cap D_{1}^{'}= \phi$. Now $(D_{1}^{'})^{c} \subset C_{x}$ and hence $(D_{1}^{'})^{c}$ is a $\mathcal{P}_{1}$-closed set such that $x \in (D_{1}^{'})^{c} \subset C_{x}$. Therefore $\mathcal{P}_{1}$-cl$(D_{1}) \subset C_{x}$ as $D_{1} \subset (D_{1}^{'})^{c} \subset C_{x}$. Again $x \notin D_{1}^{c}$, a $\mathcal{P}_{1}$-closed set and hence by pairwise regularity of $X$ there exists a $\mathcal{P}_{1}$-open set $D_{2}$ containing $x$ and a $\mathcal{P}_{2}$-open set $D_{2}^{'}$ containing $D_{1}^{c}$ with $D_{2} \cap D_{2}^{'}= \phi$. Now $D_{2} \subset (D_{2}^{'})^{c}$ and $D_{2} \subset (D_{2}^{'})^{c} \subset D_{1} \subset C_{x}$. Hence $\mathcal{P}_{2}$-cl$(D_{2}) \subset C_{x}$ and also $D_{2} \subset D_{1}$. Therefore $\mathcal{P}_{1}$-cl$(D_{2}) \subset \mathcal{P}_{1}$-cl$(D_{1})$ and hence $\mathcal{P}_{1}$-cl$(D_{2}) \cup \mathcal{P}_{2}$-cl$(D_{2}) \subset C_{x}$. Similarly if $C_{x} \in \mathcal{P}_{2}$ then there exists a $\mathcal{P}_{2}$-open set $D_{2}$ containing $x$ such that $\mathcal{P}_{1}$-cl$(D_{2}) \cup \mathcal{P}_{2}$-cl$(D_{2}) \subset C_{x}$. Let us denote $D_{2}$ by a general notation $G_{x}$ and then we can write $\mathcal{P}_{1}$-cl$(G_{x}) \cup \mathcal{P}_{2}$-cl$(G_{x}) \subset C_{x}$. Then, since $\mathcal{C}$ be a pairwise open cover $\{ G_{x} : x \in X , C_{x} \in  \mathcal{C}  \}$ is a pairwise open cover of $X$ which refines of $\mathcal{C}$. Therefore by $(iii)$ there exists a locally finite refinement $\mathcal{B}$ of $\{ G_{x} : x \in X \}$ and hence of $\mathcal{C}$. If $B \in \mathcal{B}$ then for some $G_{x}$ we have $B \subset G_{x} \subset C_{x}$ and so $\mathcal{P}_{1}$-cl$(B) \cup \mathcal{P}_{2}$-cl$(B) \subset \mathcal{P}_{1}$-cl$(G_{x}) \cup \mathcal{P}_{2}$-cl$(G_{x}) \subset C_{x}$.\\
$(iv) \Rightarrow (i)$\\
Let $\mathcal{C}$ be a pairwise open cover of $X$ and without any loss of generality we assume that there does not exists any element of $\mathcal{C}$ which is both $\mathcal{P}_{1}$-open and $\mathcal{P}_{2}$-open. So there exists a locally finite refinement $\mathcal{A}$ of $\mathcal{C}$. For $x \in X$ we must have a $C \in \mathcal{C}$ containing $x$. Let us suppose $C$ is $\mathcal{P}_{i}$-open. Let $W_{x}$ be a $\mathcal{P}_{i}$-open neighborhood of $x$ intersecting only a finite number of elements of $\mathcal{A}$. So the collection $\mathcal{W}=\{W_{x} : x\in X\}$ is a pairwise open cover of $X$ and let $E= \{E_{\lambda} : \lambda \in \wedge\}$ be a locally finite refinement of $\mathcal{W}$ such that if $E_{\lambda} \subset W_{x}$ then $\mathcal{P}_{1}$-cl$(E_{\lambda}) \cup \mathcal{P}_{2}$-cl$(E_{\lambda}) \subset W_{x}$. Now for $A \in \mathcal{A}$ we consider $C_{A} \in \mathcal{C}$ such that $A \subset C_{A}$. Then if $C_{A}$ is $\mathcal{P}_{i}$-open, then we consider the set $F_{A} =\cup \{ \mathcal{P}_{i}$-$cl(E_{\lambda}) : E_{\lambda} \in E, \mathcal{P}_{i}$-$cl(E_{\lambda}) \cap A=\phi \}$. Let $G_{A}= X \setminus F_{A}$, then by the given condition C$(2)$ there exists a $\mathcal{P}_{i}$-open set $S_{A}$ such that $A \subset S_{A} \subset G_{A}$. We write $H_{A}=S_{A} \cap C_{A}$ and since $A \subset H_{A}$ , the collection $\{H_{A} : A \in \mathcal{A} \}$ covers $X$. Also $H_{A} \subset C_{A}$ and $H_{A}$ is $\mathcal{P}_{i}$-open. Thus $\{H_{A} : A \in \mathcal{A} \}$ is a parallel refinement of $\mathcal{C}$. Now we show that $\{H_{A} : A \in \mathcal{A} \}$ is a locally finite refinement of $\mathcal{C}$. \\ 
\indent We show that if $M$ is a $\mathcal{P}_{\mathcal{W}x}$-open set containing $x$ then it is also a $\mathcal{P}_{\mathcal{C}x}$-open set containing $x$. Let $M$ is a $\mathcal{P}_{\mathcal{W}x}$-open set containing $x$ and $M$ is $\mathcal{P}_{i}$-open set then $x$ must be contained in a $\mathcal{P}_{i}$-open set $W_{x}$ in $\mathcal{W}$. So there exists a $\mathcal{P}_{i}$-open set $C$ in $\mathcal{C}$ containing $x$. This shows that $M$ is also a $\mathcal{P}_{\mathcal{C}x}$-open set containing $x$.\\ 
\indent Now let $x \in X$ and $J_{x}$ be a $\mathcal{P}_{\mathcal{W}x}$-open neighborhood of $x$ intersecting only a finite numbers of members $ E_{{\lambda}_{1}}, E_{{\lambda}_{2}} \dots E_{{\lambda}_{n}} $ of $E$. Hence $J_{x}$ is also a $\mathcal{P}_{\mathcal{C}x}$-open neighborhood of $x$ intersecting only a finite numbers of members $ E_{{\lambda}_{1}}, E_{{\lambda}_{2}} \dots E_{{\lambda}_{n}} $ of $E$.
Clearly $J_{x}$ can be covered by these members of $E$. Now each $ E_{{\lambda}_{i}}$ is contained in some $W_{x}$ with $\mathcal{P}_{1}$-cl$(E_{{\lambda}_{i}}) \cup \mathcal{P}_{2}$-cl$(E_{{\lambda}_{i}}) \subset W_{x}$. Also $W_{x}$ can intersects only a finite number of members of $\mathcal{A}$. Hence each $\mathcal{P}_{1}$-cl$(E_{{\lambda}_{i}})$ or $\mathcal{P}_{2}$-cl$(E_{{\lambda}_{i}})$ can intersect only a finite number of sets in $\mathcal{A}$. So each $\mathcal{P}_{1}$-cl$(E_{{\lambda}_{i}})$ or $\mathcal{P}_{2}$-cl$(E_{{\lambda}_{i}})$ can intersect only a finite number of sets in $\{G_{A} : A \in \mathcal{A} \}$. Therefore $J_{x}$ can intersect only a finite number of sets of $\{ G_{A} : A \in \mathcal{A}\}$. Now $\{H_{A} : A \in \mathcal{A}\}$ covers $X$ and $H_{A} \subset G_{A}$, hence $J_{x}$ can intersect only a finite number of sets in $\{H_{A} : A \in \mathcal{A}\}$. Also $H_{A} \subset C_{A}$ and hence clearly $\{H_{A} : A \in \mathcal{A}\}$ refines $\mathcal{C}$. Therefore $\{H_{A} : A \in \mathcal{A}\}$ is a locally finite parallel refinement of $\mathcal{C}$.
\end{proof}
\begin{thm}
Let $\mathcal{A}$ be a locally finite collection in a $\sigma$-space $X$. Then the collection $\mathcal{B} =\{ \overline{A} \}_{A \in \mathcal{A}}$ is also locally finite.
\end{thm}
\begin{proof}
Let $x \in X$ and $U$ be a neighborhood of $x$ intersecting only a finite number of members of $\mathcal{A}$. Now if for $A \in \mathcal{A}$, $A \cap U = \phi$ then $A \subset U^{c}$ and hence $A \subset \overline{A} \subset U^{c}$. Therefore $\overline{A} \subset U^{c}$ so $\overline{A} \cap U= \phi$. Therefore $U$ can intersect only a finite number of members of $\mathcal{B}$. 
\end{proof}
\begin{thm}
In a space any sub collection of a locally finite collection of sets is locally finite. 
\end{thm}
\begin{proof}
Let $\mathcal{A}$ be a locally finite collection of sets in a space $X$ and $\mathcal{B} = \{ B_{\alpha}  : \alpha \in \Lambda, \text{an indexing set}\}$ be a sub collection of $\mathcal{A}$. If $x \in X$ then there exists a neighborhood $U$ of $X$ intersecting only a finite number of sets in $\mathcal{A}$. Hence $U$ can not intersect infinite number of sets in $\mathcal{B}$. If $U$ does not intersect any member of $\mathcal{B}$, then consider $B_{p} \in \mathcal{B}$ such that $M=B_{p} \setminus \bigcup _{\alpha \in \Lambda} ^{\alpha \neq p} B_{\alpha}\neq \phi$. Then $M \cup U$ is a neighborhood of $x$ intersecting only $B_{p}$ of $\mathcal{B}$. Hence $\mathcal{B}$ is locally finite.
\end{proof}
It has been discussed in \cite{DJ66} that in a regular topological space $X$ the following four conditions are equivalent:\\
$(i)$ The space $X$ is paracompact.\\
$(ii)$ If $\mathcal{U}$ is a open cover of $X$ then it has an open refinement $\mathcal{V}=\bigcup_{n=1} ^{\infty} V_{n}$, where $V_{n}$ is a locally finite collection in $X$ for each $n$.\\ 
$(iii)$For every open cover of the space $X$ there exists a locally finite refinement of it.\\
$(iv)$For every open cover of the space $X$ there exists a closed locally finite refinement of it.\\
In a $\sigma$-space it is not true because closure of a set may not be closed. But a similar kind of result has been discussed below.
\begin{thm}
In a regular space $X$ for the following four conditions we have $(i)\Rightarrow (ii) \Rightarrow (iii) \Rightarrow (iv)$:\\
$(i)$ The space $X$ is paracompact.\\
$(ii)$ If $\mathcal{U}$ is a open cover of $X$ then it has an open refinement $\mathcal{V}=\bigcup_{n=1} ^{\infty} V_{n}$, where $V_{n}$ is a locally finite collection in $X$ for each $n$.\\ 
$(iii)$For every open cover of the space $X$ there exists a locally finite refinement of it.\\
$(iv)$For every open cover $\mathcal{A}$ of the space $X$ there exists a locally finite refinement $S= \{ S_{\alpha} : \alpha \in \Lambda \}$ of it such that $\{ \overline{S_{\alpha}} : S_{\alpha} \in S\}$ is also a locally finite refinement of it, $\Lambda$ being an indexing set.\\
\end{thm}
\begin{proof}
$(i)\Rightarrow (ii)$\\
The proof is straightforward.\\
$(ii) \Rightarrow (iii)$\\
Let $\mathcal{A}$ be an open cover of $X$. Then by $(ii)$ there exists an open refinement $\mathcal{B}=\bigcup_{n=1} ^{\infty} B_{n}$ where $B_{n}$ is a locally finite collection in $X$ for each $n$. Let $B_{n}=\{B_{n \alpha} : \alpha \in \Lambda_{n} \}$ and $C_{n}= \bigcup _{\alpha \in \Lambda_{n}} B_{n \alpha}$, $\Lambda_{n}$ being an indexing set. Now clearly the collection $\{ C_{n} \}$ covers $X$. Let us consider $D_{n} = C_{n} \setminus \bigcup _{k <n} C_{k}$. For $x \in X$, suppose that $k$ be the least natural number for which $x \in B_{k \alpha}$, then $B_{k \alpha}$ can intersect at most $k$ members $D_{1}, D_{2} \dots D_{k}$ of $\{ D_{n}: n\in \mathbb{N} \}$. Hence $\{ D_{n}: n\in \mathbb{N} \}$ is a locally finite refinement of $\{ C_{n}: n\in \mathbb{N} \}$. Now we show that $M= \{ D_{n} \cap B_{n \alpha} : n \in \mathbb{N}, \alpha \in \Lambda_{n} \}$ is a locally finite refinement of $\mathcal{B}$. For $n \in \mathbb{N}$ we have $\bigcup _{\alpha \in \Lambda_{n}}(D_{n} \cap B_{n \alpha}) =D_{n} \cap (\bigcup _{\alpha \in \Lambda_{n}}B_{n \alpha}) =D_{n} \cap C_{n} =D_{n}$ as $D_{n} \subset C_{n}$. Also $D_{n}$ covers $X$ and hence $\bigcup_{n \in \mathbb{N}}\bigcup _{\alpha \in \Lambda_{n}}(D_{n} \cap B_{n \alpha} ) =X$. Let $x \in X$ then there exists an  neighborhood $U$ of $x$ intersecting only a finite number members $D_{i_{1}}, D_{i_{2}} \dots D_{i_{n}}(say)$ of $\{ D_{n}: n\in \mathbb{N} \}$. Also there exists an open set $U_{i_{n}}$ intersecting only a finite number of members of $B_{i_{n}}$. Now $U \cap (\bigcap_{k=1}^{n}U_{i_{k}})$ is an neighborhood of $x$ intersecting only a finite numbers of $M$ as $M$ covers $X$. Also $D_{n} \cap B_{n \alpha} \subset B_{n \alpha}$ and hence $M= \{ D_{n} \cap B_{n \alpha} : n \in \mathbb{N}, \alpha \in \Lambda_{n} \}$ is a locally finite refinement of $\mathcal{B}$. And also since $D_{n} \cap B_{n \alpha} \subset B_{n \alpha} \subset A$ for some $A \in \mathcal{A}$, $M= \{ D_{n} \cap B_{n \alpha} : n \in \mathbb{N}, \alpha \in \Lambda_{n} \}$ is a locally finite refinement of $\mathcal{A}$.\\
$(iii) \Rightarrow (iv)$\\
Let $\mathcal{U}$ be an open cover of $X$. Now for $x \in X$ we have a $U_{x} \in \mathcal{U}$ such that $x \in U_{x}$. So $x \notin (U_{x})^{c}$ and hence by regularity of $X$, there exist disjoint open sets $P_{x}$ and $Q_{x}$ containing $x$ and $(U_{x})^{c}$ respectively.  Hence $x \in P_{x} \subset (Q_{x})^{c} \subset U_{x}$ and clearly $x \in \overline{P_{x}} \subset U_{x}$. Now $P=\{ P_{x}: x\in X \}$ is an open cover of $X$ and by $(iii)$ it has a locally finite refinement $S= \{ S_{\alpha} : \alpha \in \Lambda, \text{an indexing set} \} \text{(say)}$. Also the collection $\{\overline{S_{\alpha}} : S_{\alpha} \in S\}$ is locally finite by previous lemma. Now for $\alpha \in \Lambda$, $S_{\alpha}\subset P_{x} \subset U_{x}$ for some $P_{x} \in P$ and hence $\overline{S_{\alpha}} \subset \overline{P_{x}} \subset U_{x}$ for some $U_{x} \in \mathcal{U}$. Therefore $S$ is a locally finite refinement of $\mathcal{U}$ such that $\{\overline{S_{\alpha}} : S_{\alpha} \in S \}$ is also a locally finite refinement of $\mathcal{U}$.\\
\end{proof}
We have discussed some results associated with paracompactness in a $\sigma$-space because our motivation was to establish the statement ``If $(X, \mathcal{P}_{1}, \mathcal{P}_{2})$ is a pairwise paracompact bispace with $(X, \mathcal{P}_{2})$ regular, then every $\mathcal{P}_{1}$-$F_{\sigma}$ proper subset is $\mathcal{P}_{2}$ paracompact''. This has been discussed in a bitopological space \cite{MKBOSE}. But we have failed due to the fact that arbitrary union of open sets in a $\sigma$-space may not be open.

\end{document}